\documentclass[11pt,reqno]{amsart}
\usepackage{amsthm,amssymb,amsmath}
\usepackage{textcomp}
\usepackage{graphicx}
\usepackage{subcaption}

\usepackage[T1]{fontenc}
\usepackage{lmodern}

\newtheorem{theorem}{Theorem}
\newtheorem*{theorem*}{Theorem}
\newtheorem{lemma}{Lemma}

\newtheorem{proposition}{Proposition}
\newtheorem{claim}{Claim}

\theoremstyle{definition}

\newtheorem*{definition*}{\bf Definition}

\newtheorem*{remark*}{Remark}

\newtheorem*{example*}{\bf Example}

\newcommand{\loc}{{\rm loc}}

\newcommand{\Real}{{\rm Re}}

\makeatletter
\@namedef{subjclassname@2020}{%
  \textup{2020} Mathematics Subject Classification}
\makeatother

\begin{document}

\title[Supercritical fractional diffusion equation]{On the supercritical fractional diffusion equation with Hardy-type drift}

\author{D.\,Kinzebulatov, K.R.\,Madou, Yu.\,A.\,Sem\"{e}nov}

\keywords{Non-local operators, heat kernel estimates, desingularization}

\subjclass[2020]{35K08, 47D07 (primary), 60J35 (secondary)}

\begin{abstract}
We study the heat kernel of the supercritical fractional diffusion equation with the drift in the critical H\"{o}lder space. We show that such a drift can have point irregularities strong enough to make the heat kernel vanish at a point for all $t>0$.
\end{abstract}

\thanks{The research of D.K.\,is supported by the Natural Sciences and Engineering Research Council 
of Canada (grant RGPIN-2017-05567)}

\address{Universit\'{e} Laval, D\'{e}partement de math\'{e}matiques et de statistique, 1045 av.\,de la M\'{e}decine, Qu\'{e}bec, QC, G1V 0A6, Canada}

\email{damir.kinzebulatov@mat.ulaval.ca}

\address{Universit\'{e} Laval, D\'{e}partement de math\'{e}matiques et de statistique, 1045 av.\,de la M\'{e}decine, Qu\'{e}bec, QC, G1V 0A6, Canada}

\email{kodjo-raphael.madou.1@ulaval.ca}

\address{University of Toronto, Department of Mathematics, 40 St.\,George Str, Toronto, ON, M5S 2E4, Canada}

\email{semenov.yu.a@gmail.com}

\maketitle

\section{Introduction and main result}

\label{main_sect}

\textbf{1.~}The present paper concerns the fractional diffusion equation 
\begin{equation}
\label{eq}
\partial_t u + (-\Delta)^{\frac{\alpha}{2}}u - \mathsf{f} \cdot \nabla u=0, \quad \mathsf{f}:\mathbb R^d \rightarrow \mathbb R^d, \quad d \geq 3
\end{equation}
in the critical ($\alpha=1$) and the supercritical  regimes ($0<\alpha<1$). The terminology ``critical'' 
and ``supercritical'' refers to the fact that when $\alpha=1$ the drift term $\mathsf{f} \cdot \nabla$ is of the same weight as the diffusion term $(-\Delta)^{\frac{\alpha}{2}}$,
while if $\alpha<1$ then, formally, $\mathsf{f} \cdot \nabla$ dominates $(-\Delta)^{\frac{\alpha}{2}}$, so the standard perturbation-theoretic techniques are not applicable.

 This equation continues to attract interest, 
motivated, in particular, by applications in hydrodynamics. 
In the supercritical regime, it was studied by 
Constantin-Wu \cite{CW} who established H\"{o}lder continuity of solution $u$ assuming 
that the vector field $\mathsf{f}$  is in $C^{0,1-\alpha}$
and ${\rm div\,}\mathsf{f}=0$. Later the H\"{o}lder continuity of solution without the 
divergence-free assumption on the drift was established by Silvestre \cite{S}. 
The  H\"{o}lder continuity exponent $1-\alpha$  arises in both papers from the scaling arguments (in a variant of the De Giorgi method and a comparison principle, respectively). Maekawa-Miura \cite{MM} considered \eqref{eq}, in particular in the supercritical regime, and established an upper bound on the heat kernel when $\mathsf{f} \in C^{0,1-\alpha}$, ${\rm div\,}\mathsf{f}=0$. 
Recently, Menozzi-Zhang \cite{MZ} established two-sided heat kernel bound in the ``sub''-supercritical case $|\mathsf{f}| \in C^{0,\gamma}$, $\gamma>1-\alpha$; Zhao \cite{Z} established weak well-posedness for the SDE associated with \eqref{eq} provided that $\|\mathsf{f}\|_{C^{0,1-\alpha}}$ is sufficiently small; see also \cite{XZ, ZZ}. (In fact, \cite{MM,MZ} allow time-dependent coefficients that can grow at infinity, \cite{MZ,XZ, Z,ZZ} deal with more general than $(-\Delta)^{\frac{\alpha}{2}}$ diffusion term.) See also references therein.

Below we show that the class $C^{0,1-\alpha}$ contains vector fields
that have point irregularities strong enough to make the heat kernel of \eqref{eq} vanish  (in the $y$ variable,  for all $t>0$). More precisely, we consider as the drift $\mathsf{f}$ a bounded, infinitely differentiable outside of the origin vector field $b:\mathbb R^d \rightarrow \mathbb R^d$ such that
\begin{equation}
\label{hardy}
b(x)=\kappa|x|^{-\alpha}x \quad \text{ in } \{|x|<1\}
\end{equation}
where $\kappa>0$. 
The vector field $b$ is, in a sense, a prototypical representative of the class $C^{0,1-\alpha}$. 
We establish a vanishing upper bound on the heat kernel, see Theorem \ref{thm}.

In order to keep the paper short, we will be assuming that on $\{|x| \geq 1\}$ the derivatives of the vector field $b$ are uniformly bounded, and $|{\rm div\,}b|$ is less than $C|x|^{-\alpha}$ for some constant $C>0$ (e.g.\,$b$ can have compact support). The method of the paper can handle $b(x)=\kappa|x|^{-\alpha}x$, $x \in \mathbb R^d$.

The critical regime $\alpha=1$, with $\mathsf{f}$ in ${\rm BMO}$ and divergence-free,  was studied by Caffarelli-Vasseur \cite{CV} and, later, by Kiselev-Nazarov \cite{KN}. The critical regime without the divergence-free condition but assuming that $|\mathsf{f}| \in L^\infty$ was considered by Silvestre \cite{S}. Our result includes $\alpha=1$ as well.

Set
$\gamma(s):=\frac{2^s\pi^\frac{d}{2}\Gamma(\frac{s}{2})}{\Gamma(\frac{d}{2}-\frac{s}{2})}.$

\begin{theorem}
\label{thm}
 Let $d \geq 3$, $0<\alpha \leq 1$. Let $b$ be defined by \eqref{hardy} with $\kappa>0$.
Then the heat kernel of the operator $\Lambda=(-\Delta)^{\frac{\alpha}{2}} - b \cdot \nabla$,
constructed in Proposition \ref{prop1} below, determines a $C_0$ semigroup in $L^r=L^r(\mathbb R^d)$ for all $r \in [1,\infty[$,  and satisfies for all $0<t \leq 1$, $x$, $y \in \mathbb R^d$
\begin{equation}
\label{ub}
0 \leq e^{-t\Lambda}(x,y) \leq Ct^{-\frac{d}{\alpha}} \bigl[1 \wedge t^{-\frac{\beta}{\alpha}}|y|^\beta \bigr]
\end{equation}
(possibly after a modification on a measure zero set in $\mathbb R^d \times \mathbb R^d$), where the order of vanishing $\beta \in ]0,\alpha[$ is determined from the equation
\begin{equation}
\label{beta_eq}
\beta\frac{d+\beta-2}{d+\beta-\alpha} \frac{\gamma(d+\beta-2)}{\gamma(d+\beta-\alpha)}=\kappa \quad (\text{see Fig.\,1}).
\end{equation}
\end{theorem}

The equation \eqref{beta_eq} is the condition that $|x|^\beta$ is the Lyapunov function of the formal operator $(-\Delta)^{\frac{\alpha}{2}} + \nabla \cdot \kappa|x|^{-\alpha}x$, i.e $\big[(-\Delta)^{\frac{\alpha}{2}} + \nabla \cdot \kappa|x|^{-\alpha}x\big]|x|^{\beta}=0.$

Theorem \ref{thm} is proved by considering operator $\Lambda_r$ in the weighted space $L^1(\mathbb R^d,\psi dx)$, with appropriate vanishing weight $\psi(x) \approx (1 \wedge |x|)^\beta$, where the operator is ``desingularized'', and the semigroup $e^{-t\Lambda_r}$ is $L^1(\mathbb R^d,\psi dx) \rightarrow L^\infty$ ultracontractive. The desingularization procedure was introduced by Milman-Sem\"{e}nov to establish two-sided heat kernel bounds for the Schr\"{o}dinger operator $-\Delta + \kappa|x|^{-2}$ \cite{MS0,MS1, MS2}. The non-symmetric, non-local desingularization for $\Lambda=(-\Delta)^{\frac{\alpha}{2}} - \kappa |x|^{-\alpha}x \cdot \nabla$ in the subcritical case $1<\alpha<2$ was developed  in Kinzebulatov-Sem\"{e}nov-Szczypkowski \cite{KSS} ($\kappa<0$) and Kinzebulatov-Sem\"{e}nov \cite{KS} ($\kappa>0$). See Theorem \ref{thmB} below. 
The desingularization procedure also works in the critical $\alpha=1$ and the supercritical $\alpha<1$ regimes, as we show in this paper. 
This is  rather notable, since $\alpha \leq 1$ is known 
to present its own set of difficulties compared to $1<\alpha<2$.

It should be noted that in \cite{KS} ($1<\alpha<2$) the authors proved, using perturbation-theoretic arguments, the following two-sided heat kernel bound
\begin{equation}
\label{bd_KS}
e^{-t\Lambda}(x,y) \approx e^{-t(-\Delta)^{\frac{\alpha}{2}}}(x,y)[1 \wedge t^{-\frac{1}{\alpha}}|y|]^{\beta}
\end{equation}
for $\beta \in ]0,\alpha[$ determined by \eqref{beta_eq}. The bound \eqref{ub} describes the behaviour of the heat kernel around the singularity of the drift, but it leaves open the question of  two-sided bound for \eqref{eq} with $\mathsf{f}:=b$ in the critical and the supercritical regimes. We plan to address it in the future.

The case of $b$ with $\kappa<0$, which corresponds to the attracting drift, can be treated by modifying the argument in \cite{KSS}:
$$
e^{-t\Lambda}(x,y) \leq Ct^{-\frac{d}{\alpha}} \bigl[1 \wedge t^{-\frac{\beta}{\alpha}}|y|^\beta\bigr] \quad \text{for $\beta \in ]-d+\alpha,0]$ such that $\Lambda^*|x|^{\beta}=0$}.
$$ 
We will not be proving this bound here (in fact, to make this result complete one has to prove the lower bound). Let us only mention that the construction of the heat kernel requires an energy inequality in some $L^r$, $r \geq 2$ (see \cite{KSS}), which imposes a constraint from below on the admissible values of $\kappa<0$ (cf.\,\cite{Z}). Namely, multiplying the equation by $u|u|^{r-2}$ and integrating, we have 
$$
\frac{2}{r}\partial_t\langle|u|^r\rangle - \lambda \langle |u|^r \rangle + \Real\langle (-\Delta)^{\frac{\alpha}{2}}u,u|u|^{r-2}\rangle - |\kappa|\frac{d-\alpha}{r} \langle |x|^{-\alpha},|u|^r \rangle \leq 0, \quad \text{ for some } \lambda>0.
$$
Now, applying the fractional Hardy inequality 
$$\Real\langle (-\Delta)^{\frac{\alpha}{2}}u,u|u|^{r-2}\rangle \geq c_{d,\alpha,r}\langle |x|^{-\alpha},|u|^r \rangle$$ with the sharp constant $c_{d,\alpha,r}$ (see \cite{BJLP}), we arrive at the condition $|\kappa|\frac{d-\alpha}{r}<c_{d,\alpha,r}$, which yields a constraint on $\kappa<0$ from below.
 In fact, in the local case $\alpha=2$, some aspects 
of the regularity theory of the corresponding parabolic equation depend on this constraint, see \cite{KS2}. 
It is interesting 
to note that, for $\alpha<2$, for every $\kappa<0$ there exists a $\beta \in ]-d,0[$ such that $\Lambda^*|x|^{\beta}=0$. (In principle, this open up a possibility to verify accretivity of $\Lambda$ in the weighed space $L^1(\mathbb R^d,\psi dx)$, $\psi(x) \approx (1 \wedge |x|)^\beta$, for any $\kappa<0$, and hence to construct a $C_0$ semigroup there. We plan to address this matter in detail elsewhere.)

In the subcritical regime $1<\alpha<2$ there is a greater variety of classes of admissible drifts having critical-order singularities. In particular, Bogdan-Jakubowski \cite{BJ} established two-sided heat kernel bounds for \eqref{eq} with $\mathsf{f}$ in the Kato class.
Regarding the case ${\rm div\,}\mathsf{f}=0$, see Jakubowski \cite{J}, Maekawa-Miura \cite{MM} who considered $\mathsf{f}$  in the Campanato-Morrey class. The weak solvability and the Feller property for the corresponding SDE with drift $\mathsf{f}$ in an even larger class of weakly form-bounded vector fields were proved in Kinzebulatov-Madou \cite{KM}.

\begin{figure}
\begin{center}
\includegraphics[width=0.9\textwidth]{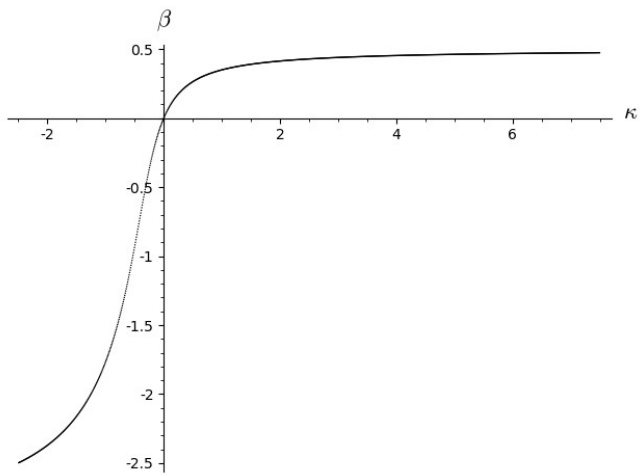}
\end{center}
\vspace*{-15mm}
\caption{The graph of $\beta$ as a function of the coefficient $\kappa$ for $d=3$ and $\alpha=\frac{1}{2}$. \label{fig1}}
\end{figure}

\medskip

\textbf{2.~}Let us describe the construction of the heat kernel in Theorem \ref{thm}. Put  $|x|_\varepsilon:=\sqrt{|x|^2+\varepsilon}$.
Let us fix smooth vector fields
$b \in C_b(\mathbb R^d) \cap C^\infty(\mathbb R^d)$, $\varepsilon>0$ such that
$$
b_\varepsilon(x):=\left\{
\begin{array}{ll}
b(x), & |x|>2,\\
\kappa|x|_\varepsilon^{-\alpha}x, & |x|<1.
\end{array}
\right.
$$
In $\{1 \leq |x| \leq 2\}$, we require uniform convergence
$$
b_\varepsilon \rightarrow b, \quad \nabla_{x_i} b_\varepsilon \rightarrow  \nabla_{x_i} b, \quad  \nabla^2_{x_ix_j} b_\varepsilon \rightarrow  \nabla^2_{x_ix_j} b
$$
and $|\nabla_{x_i}b_\varepsilon| \leq \sigma_1$ ($i=1,\dots,d$), $|{\rm div\,}b_\varepsilon| \leq \sigma_2$ on $\{|x| \geq 1\}$ with constants $\sigma_1$, $\sigma_2$ independent of $\varepsilon$.

For $r \in [1,\infty[$ put
 $$\Lambda_r^\varepsilon:=-\varepsilon \Delta + (-\Delta)^{\frac{\alpha}{2}} - b_\varepsilon \cdot \nabla, \quad D(\Lambda_r^\varepsilon)=\mathcal W^{2,r} \quad (\text{Bessel space}),$$
the generator of a positivity preserving $L^\infty$ contraction quasi contraction holomorphic semigroup (e.g.\,by the Hille Perturbation Theorem, cf.\,\cite[Sect.\,8]{KS}).

\begin{proposition}
\label{prop1} Let $d \geq 3$, $0<\alpha \leq 1$. Let $b$ be defined by \eqref{hardy} with $\kappa>0$. For every $r \in [1,\infty[$, the limit
$$
s\mbox{-}L^r\mbox{-}\lim_{\varepsilon \downarrow 0} e^{-t\Lambda_r^{\varepsilon}} \quad (\text{loc.\,uniformly in $t \in [0,\infty[$})
$$ 
exists and determines a $L^\infty$ contraction positivity preserving quasi contraction  semigroup on $L^r$, say, $e^{-t\Lambda_r}$. Its generator 
$\Lambda_r$ is an appropriate operator realization of the formal operator $(-\Delta)^{\frac{\alpha}{2}} - b \cdot \nabla$ in $L^r$. 

The Sobolev embedding property and the ultracontractivity property hold:
$$
\langle \Lambda_2 u,u \rangle \geq c_S \|u\|_{\frac{2d}{d-\alpha}}^2, \quad u \in D(\Lambda_2),
$$ 
$$
\|e^{-t\Lambda_r}\|_{r \rightarrow q} \leq c_N e^{\omega_r t }t^{-\frac{d}{\alpha}(\frac{1}{r}-\frac{1}{q})}, \quad t \in [0,\infty[, \quad 1 \leq r < q \leq \infty,
$$
where $c_S$, $c_N$ are generic constants. 

$e^{-t\Lambda_r}$ is a semigroup of integral operators. 
\end{proposition}

By construction, the integral kernel  $e^{-t\Lambda}(x,y)$ of $e^{-t\Lambda_r}$ does not depend on $r$.
It is defined to be the heat kernel of $(-\Delta)^{\frac{\alpha}{2}} - b \cdot \nabla$. One can easily see that $u(t):=e^{-t\Lambda_2}f$ with $f \in L^2$ is a weak solution to \eqref{eq}.

\subsection*{Notations}
We write
$$
\langle u,v\rangle = \langle u\bar{v}\rangle :=\int_{\mathbb R^d}u\bar{v}dx.
$$

The fractional Laplacian $(-\Delta)^{\frac{\alpha}{2}}$ is defined in $L^r$, $r \in [1,\infty[$ or $C_u$  (bounded uniformly continuous functions with the $\sup$-norm) in the sense of Balakrishnan. (Here $-\Delta$ is defined in $L^r$ or $C_u$ as the generator of the heat semigroup in these spaces.)

We denote by $\mathcal B(X,Y)$ the space of bounded linear operators between Banach spaces $X \rightarrow Y$, endowed with the operator norm $\|\cdot\|_{X \rightarrow Y}$. Set  $\mathcal B(X):=\mathcal B(X,X)$.

We write $T=s\mbox{-} X \mbox{-}\lim_n T_n$ for $T$, $T_n \in \mathcal B(X)$ if $Tf=\lim_n T_nf$ in $X$ for every $f \in X$. We also write $T_n \overset{s}{\rightarrow} T$ if $X=L^2$.

Denote $\|\cdot\|_{p \rightarrow q}:=\|\cdot\|_{L^p \rightarrow L^q}$.

We say that a constant is generic if it only depends on $d$, $\kappa$, $\alpha$, $\sigma_1$, $\sigma_2$.

\bigskip

\section{Proof of Proposition \ref{prop1}}

\label{sect_d}

The proof  is contained in the next three claims. 

\begin{claim}
\label{claim0}
For every $r \in [1,\infty[$ and all $\varepsilon>0$, 
$$
\|e^{-t\Lambda_r^{\varepsilon}}\|_{r \rightarrow r} \leq e^{\omega_r t}.
$$
There exists constant $c_N$  independent of $\varepsilon$ such that, for all $1 \leq r < q \leq \infty$,
$$
\|e^{-t\Lambda^\varepsilon_r}\|_{r \rightarrow q} \leq c_N  e^{\omega_r t} t^{-\frac{d}{\alpha}(\frac{1}{r}-\frac{1}{q})}, \quad t>0.
$$
There exists constant  $c_S$  independent of $\varepsilon$ such that
$$
\langle \Lambda_2^\varepsilon u,u \rangle \geq c_S \|u\|_{\frac{2d}{d-\alpha}}^2, \quad u \in D(\Lambda_2^\varepsilon)=\mathcal W^{2,2}.
$$
\end{claim}

Although the proof of Claim \ref{claim0} is standard, we included it in Appendix \ref{claim0_proof} for the sake of completeness.

\smallskip

To prove that $s\mbox{-}L^r\mbox{-}\lim_{\varepsilon \downarrow 0}e^{-t\Lambda_r^{\varepsilon}}$ exists and determines a $C_0$ semigroup,  we will show that $\{e^{-t\Lambda_r^{\varepsilon_n}}f\} $ is a Cauchy sequence in $L^\infty([0,1],L^r)$, for any $f \in C_c^\infty$ and any $\{\varepsilon_n\} \downarrow 0$,. 
For that, we will need a uniform bound on the $L^2$ norm of the gradient of $u^\varepsilon(t):=e^{-t\Lambda^\varepsilon}f$.

\begin{claim}
\label{claim1}
There exists a constant $\omega_3$ independent of $\varepsilon$ such that
$$\|\nabla u^\varepsilon(t)\|_2 \leq e^{t\omega_3}\|\nabla f\|_2, \quad t \geq 0.$$
\end{claim}

\begin{proof}[Proof of Claim \ref{claim1}]
Denote
$
u \equiv u^\varepsilon$, $w:=\nabla u$, $w_i:=\nabla_i u.
$
Since $f\in C^\infty_c$ and $\nabla^n_i b^i_\varepsilon\in C^\infty$ ($n=0,1$) are bounded and continuous, we can differentiate the equation $\partial_t u + \Lambda^\varepsilon u=0$ in $x_i$, obtaining
$$
\partial_t w_i -\varepsilon \Delta w_i  + (-\Delta)^{\frac{\alpha}{2}}w_i - b_\varepsilon \cdot \nabla w_i - (\nabla_i b_\varepsilon) \cdot w=0.
$$
Multiplying the latter by $\bar{w_i}$, integrating by parts and summing up in $i=1,\dots,d$, we obtain
$$
\frac{1}{2} \partial_t \|w\|_2^2 + \varepsilon \sum_{i=1}^d\|\nabla w_i\|_2^2 + \sum_{i=1}^d \|(-\Delta)^{\frac{\alpha}{4}}w_i\|_2^2 - \Real\sum_{i=1}^d\langle  b_\varepsilon \cdot \nabla w_i,w_i \rangle - \Real\sum_{i=1}^d \langle (\nabla_i b_\varepsilon) \cdot w,w_i \rangle =0.
$$
Here, using the integration by parts, we obtain
\begin{align*}
& -\Real\langle b_\varepsilon \cdot \nabla w_i,w_i\rangle  =  \frac{1}{2}\langle ({\rm div\,}b_\varepsilon) w_i,w_i\rangle  \\
& \geq \frac{\kappa}{2}\langle \mathbf{1}_{|x| < 1}(d|x|_\varepsilon^{-\alpha}-\alpha |x|_\varepsilon^{-\alpha-2}|x|^2)w_i,w_i\rangle - \frac{\sigma_2}{2}\langle w_i,w_i\rangle.
\end{align*}
Also, 
\begin{align*}
-\langle (\nabla_i b_\varepsilon) \cdot w,w_i \rangle & \geq -\kappa \langle \mathbf{1}_{|x|<1}|x|_\varepsilon^{-\alpha}w_i,w_i\rangle +\kappa \alpha \langle \mathbf{1}_{|x|<1}|x|_\varepsilon^{-\alpha-2} x_i \bar{w}_i (x \cdot w) \rangle  - \sigma_1 \langle \mathbf{1}_{|x| \geq 1} |w|^2\rangle,
\end{align*}
and so 
$$
- \Real\sum_{i=1}^d \langle (\nabla_i b_\varepsilon) \cdot w,w_i \rangle \geq -\kappa \langle \mathbf{1}_{|x|<1}|x|_\varepsilon^{-\alpha}|w|^2\rangle - \sigma_1  d\langle  |w|^2\rangle.
$$

Thus, 
\begin{align*}
 & \frac{1}{2} \partial_t \|w\|_2^2  +  \varepsilon \sum_{i=1}^d\|\nabla w_i\|_2^2 + \sum_{i=1}^d \|(-\Delta)^{\frac{\alpha}{4}}w_i\|_2^2 \\
& + \kappa \frac{ d-\alpha-2  }{2}\langle \mathbf{1}_{|x|<1}|x|_\varepsilon^{-\alpha}|w|^2\rangle +\frac{\kappa\alpha\varepsilon}{2}\langle  \mathbf{1}_{|x|<1} |x|_\varepsilon^{-\alpha-2}|w|^2\rangle - \bigl(\sigma_1 d+\frac{\sigma_2}{2} \bigr) \|w\|_2^2  \leq 0,
\end{align*}
and so, since $\kappa>0$,
\[
\frac{1}{2} \partial_t \|w\|_2^2 + \kappa \frac{ d-\alpha -2 }{2}\langle \mathbf{1}_{|x|<1}|x|_\varepsilon^{-\alpha}|w|^2\rangle  \leq \bigl(\sigma_1 d+\frac{\sigma_2}{2} \bigr) \|w\|_2^2.
\]
Since $d \geq 3$, $\alpha \leq 1$, we have $ d-\alpha-2 \geq 0$. Thus, integrating in $t$, we obtain
$$
\|w(t)\|_2^2 \leq e^{t\omega_3}\|\nabla f\|_2^2, \quad t \geq 0, \quad \omega_3:=\sigma_1 d+\frac{\sigma_2}{2}.
$$
\end{proof}

Next, set $u_n:=u^{\varepsilon_n}$, $b_n:=b_{\varepsilon_n}$, where $\varepsilon_n \downarrow 0$, and put $$g(t):=u_n(t) - u_m(t), \quad t\geq 0.$$

\begin{claim}
\label{claim2}
$\|g(t)\|_2 \rightarrow 0$ uniformly in $t \in [0,1]$ as $n,m \rightarrow \infty$.
\end{claim}

\begin{proof}[Proof of Claim \ref{claim2}]
We subtract the equations for $u_n$ and $u_m$ and obtain
$$
\partial_t g  - \varepsilon_{n} \Delta g - (\varepsilon_n-\varepsilon_m) \Delta u_m  + (-\Delta)^{\frac{\alpha}{2}}g - b_n \cdot \nabla g - (b_n-b_m) \cdot \nabla u_m=0,
$$
so, after multiplying by $g$ and integrating, we have
\begin{align}
\label{id_g}
\frac{1}{2}\partial_t \|g\|_2^2 & + \varepsilon_n \|\nabla g\|_2^2 + (\varepsilon_n-\varepsilon_m) \langle \nabla u_m,\nabla g \rangle \notag \\ 
& + \|(-\Delta)^{\frac{\alpha}{4}}g\|_2^2 
 - \Real\langle b_n \cdot \nabla g, g\rangle - \Real \langle (b_n-b_m) \cdot \nabla u_m,g\rangle=0.
\end{align}
Concerning the last two terms, we have (uniformly in $t \in [0,1]$):
$$
-\Real \langle b_n \cdot \nabla g, g\rangle \geq - \frac{\sigma_2}{2}\|g\|_2^2
$$
(arguing as in the proof of Claim \ref{claim1}), and
\begin{align*}
|\langle (b_n-b_m) \cdot \nabla u_m,g\rangle| & = |\langle \mathbf{1}_{|x|<2}(b_n-b_m) \cdot \nabla u_m,g\rangle|\\
& \text{(we use $\|g\|_\infty \leq 2\|f\|_\infty$)} \\
& \leq \|\mathbf{1}_{|x|<2} (b_n-b_m)\|_2 \|\nabla u_m\|_2 2 \|f\|_\infty \\
& (\text{we are using Claim \ref{claim1}}) \\
& \leq 2e^{\omega_3}\|\mathbf{1}_{|x|<2} (b_n-b_m)\|_2 \|\nabla f\|_2 \|f\|_\infty  \rightarrow 0 \quad \text{ as $n,m \rightarrow \infty$},
\end{align*}
Using again Claim \ref{claim1}, we have
$$
|(\varepsilon_n-\varepsilon_m) \langle \nabla u_m,\nabla g \rangle| \leq |\varepsilon_n-\varepsilon_m| \|\nabla u_m\|_2\|\nabla g\|_2 \rightarrow 0 \qquad \text{as $n,m \rightarrow \infty$}.
$$
Thus, integrating \eqref{id_g} in $t$ and using the last three observations, we have for all $0<\tau \leq 1$
$$
\sup_{t \in [0,\tau]}\|g(t)\|_2^2 - \sigma_2\int_0^\tau \|g(s)\|_2^2 ds \leq o(\varepsilon),
$$
where $o(\varepsilon) \rightarrow 0$ as $\varepsilon \downarrow 0$. It follows that
$$
(1-\sigma_2 \tau)\sup_{t \in [0,\tau]}\|g(t)\|_2^2 \leq o(\varepsilon),
$$
where $\tau>0$ is fixed so that $\sigma_2\tau<1$.
This yields the required convergence on $[0,\tau]$. Now, the latter and the reproduction property of the approximating semigroups end the proof of Claim \ref{claim2}.
\end{proof}

By Claim \ref{claim2}, $\{e^{-t\Lambda^{\varepsilon_n} }f\}_{n=1}^\infty$, $f \in C_c^\infty$ is a Cauchy sequence in $L^\infty([0,1],L^2)$. Set
\begin{equation*}
T_2^t f:=s\mbox{-}L^2\mbox{-}\lim_n e^{-t\Lambda^{\varepsilon_n}}f \text{ uniformly in }0 \leq t \leq 1.
\end{equation*}
(Clearly, the limit does not depend on the choice of $\{\varepsilon_n\}\downarrow 0$.) Extending $T_2^t$ by continuity to $L^2$, and then to all $t>0$ by postulating the reproduction property, we obtain a $C_0$ semigroup on $L^2$. Put $e^{-t\Lambda_2}:=T_2^t$, $t \geq 0$.
Now Claim \ref{claim0} and the standard density argument yields convergence in all $L_r$, $1 \leq r<\infty$. 
The ultracontractivity property follows. The fact that the resulting semigroups are integral operators is an immediate consequence of the ultracontractivity and the Dunford-Pettis Theorem.

It remains to prove the Sobolev embedding property. By Claim \ref{claim0}  ($\Lambda^\varepsilon \equiv \Lambda^\varepsilon_2$), 
$$
\Real\big\langle \Lambda^\varepsilon(1+\Lambda^\varepsilon)^{-1}g,(1+\Lambda^\varepsilon)^{-1}g\big\rangle \geq c_S \|(1+\Lambda^\varepsilon)^{-1}g\|_{\frac{2d}{d-\alpha}}^2, \quad g \in L^2, \quad c_S \neq c_S(\varepsilon),
$$
i.e.
$$
\Real\big\langle g-(1+\Lambda^\varepsilon)^{-1}g,(1+\Lambda^\varepsilon)^{-1}g\big\rangle \geq c_S \|(1+\Lambda^\varepsilon)^{-1}g\|_{\frac{2d}{d-\alpha}}^2.
$$
Using the convergence $(1+\Lambda^\varepsilon)^{-1} \overset{s}{\rightarrow } (1+\Lambda)^{-1}$ in $L^2$ as $\varepsilon \downarrow 0$, we obtain $\Real\big\langle \Lambda(1+\Lambda)^{-1}g,(1+\Lambda)^{-1}g\big\rangle \geq c_S \|(1+\Lambda)^{-1}g\|_{\frac{2d}{d-\alpha}}^2$ for all $g \in L^2$, and so the Sobolev embedding follows.

\bigskip

\section{Proof of Theorem \ref{thm}}

\subsection{Desingularization theorem}

We first state an abstract desingularization theorem from \cite{KS}. We will apply it in the next section to the operator $(-\Delta)^{\frac{\alpha}{2}} - b \cdot \nabla$.

Let $X$ be a locally compact topological space, and $\mu$ a $\sigma$-finite Borel measure on $X$. Set $L^p=L^p(X,\mu)$, $p \in [1,\infty]$, a (complex) Banach space. Let $\|\cdot\|_{p \rightarrow q}:=\|\cdot\|_{L^p \rightarrow L^q}$.
Let $-\Lambda$ be the generator of a contraction $C_0$ semigroup $e^{-t\Lambda}$, $t>0$, in $L^2$.

Assume that, for some constants $M\geq 1$, $c_S>0$, $j>1$, $c>0$,
\[
\|e^{-t\Lambda}f\|_1\leq M\|f\|_1 , \quad t\geq 0, \quad f\in L^1\cap L^2 .\tag{$B_{11}$}
\]
\[
\text{Sobolev embedding property:} \quad \Real\langle\Lambda u,u\rangle\geq c_S\|u\|^2_{2j}, \quad u\in D(\Lambda).\tag{$B_{12}$}
\]
\[
\|e^{-t\Lambda}\|_{2 \rightarrow \infty} \leq ct^{-\frac{j'}{2}}, \quad t>0, \quad j'=\frac{j}{j-1}.\tag{$B_{13}$}
\]

\smallskip

Assume also that there exists a family of real-valued weights $\psi=\{\psi_s\}_{s>0}$ on $X$ such that, for all $s>0$, 
\[
0 \leq \psi_s, \psi_s^{-1} \in L^1_{\loc}(X - N,\mu), \quad \text{where } N \text{ is a closed null set},\tag{$B_{21}$}
\]
and there exist constants $\theta \in ]0,1[$, $\theta \neq \theta (s)$, $c_i\neq c_i(s)$ {\rm($i=2,3$)} and a measurable set $\Omega^s \subset X$ such that
\[
\psi_s(x)^{-\theta } \leq c_2 \text{ for all } x\in X-\Omega^s, \tag{$B_{22}$}
\]
\[
\|\psi_s^{-\theta }\|_{L^{q^\prime}(\Omega^s)}\leq c_3 s^{j^\prime /q^\prime},\text{ where } q^\prime=\frac{2}{1-\theta}.\tag{$B_{23}$}
\]

\begin{theorem}[{\cite[Theorem 1]{KS}}]
\label{thmB}
In addition to $(B_{11})-(B_{23})$ assume that there exists a constant $c_1\neq c_1(s)$ such that, for any $s>0$ and all $\frac{s}{2}\leq t \leq s$,
\[
\|\psi_s e^{-t\Lambda}\psi_s^{-1}f\|_{1} \leq c_1\|f\|_{1}, \quad f \in L^1.
\tag{$B_3$}
\]
Then there is a constant $C$ such that, for all $t>0$ and $\mu$ a.e. $x,y \in X$,
\[
|e^{-t\Lambda}(x,y)|\leq C t^{-j^\prime}\psi_t(y).
\]  
\end{theorem}

Theorem \ref{thmB} is a weighted Nash initial estimate \cite{N}.

\subsection{Proof of Theorem \ref{thm}}
Define weights $\psi_t \in C^2(\mathbb R^d - \{0\}) \cap C_b(\mathbb R^d)$ by $$\psi_t(y)=\eta(t^{-\frac{1}{\alpha}}|y|), \quad y \in \mathbb R^d,$$ where
\[
\eta(\tau)=\left \{
\begin{array}{ll}
\tau^\beta, & 0<\tau< 1, \\
\beta \tau(2-\frac{\tau}{2})+1-\frac{3}{2}\beta, &1\leq \tau < 2, \\
1+\frac{\beta}{2}, & \tau \geq 2
\end{array}
\right.
\]
(the constant $\beta$ is determined from the equation \eqref{beta_eq}).

Theorem \ref{thm} will follow from Theorem \ref{thmB} applied to the semigroup $e^{-t\Lambda} \equiv e^{-t\Lambda_2}$, $\Lambda_2 \supset (-\Delta)^{\frac{\alpha}{2}} - b \cdot \nabla$, which was constructed in Proposition \ref{prop1}. Thus, we will prove that for all $t \in [0,1]$, for a.e.\,$x$, $y \in \mathbb R^d$, 
$$
e^{-t\Lambda}(x,y) \leq Ct^{-\frac{d}{\alpha}}\psi_t(y),
$$
which yields Theorem \ref{thm}.

In Proposition \ref{prop1} we proved that $e^{-t\Lambda}$ satisfies conditions ($B_{11}$), ($B_{12}$) and ($B_{13}$) with $j'=\frac{d}{\alpha}$.
The condition ($B_{21}$) is evident. It is easily seen that $(B_{22}),(B_{23})$ hold with $$\Omega^s=B(0,s^{\frac{1}{\alpha}}), \quad \theta=\frac{(2-\alpha)d}{(2-\alpha)d+8\beta}.$$

It remains to verify ($B_3$). This step presents the main difficulty.
We will show that $\psi_s e^{-t\Lambda}\psi_s^{-1}$ is a quasi contraction semigroup in $L^1$, i.e.\,there exists $\hat{c}>0$ such that for any $s>0$
\begin{equation}
\label{accr0}
\|\psi_s e^{-t\Lambda}\psi_s^{-1}f\|_1 \leq e^{(\hat{c}s^{-1}+\sigma_2)t}\|f\|_1, \quad t>0.
\end{equation}
Then, taking $\frac{s}{2}\leq t \leq s$ and $t \in [0,1]$, we obtain ($B_3$). 

Intuitively, the generator of $\psi_s e^{-t\Lambda}\psi_s^{-1}$ should be $\psi_s \Lambda_1 \psi_s^{-1}$. Thus, it would suffice to show that  $\lambda + \psi_s \Lambda_1 \psi_s^{-1}$ is accretive in $L^1$ for some $\lambda>0$, i.e.\,formally, for all admissible $f$,
$$
\big\langle (\lambda + \psi_s \Lambda_1 \psi_s^{-1})f,\frac{f}{|f|}\big\rangle \geq 0.
$$
 However, a direct calculation is problematic:  $\Lambda_1$ is not an algebraic sum of $(-\Delta)_{L_1}^{\frac{\alpha}{2}}$ and $(b \cdot \nabla)_{L^1}$, there is no explicit description of the domain $D(\Lambda_1)$ and, furthermore, $\psi_s^{-1}$ is unbounded.
Instead, we will carry out an approximation argument, replacing $\Lambda_1$ by the approximating operators $\Lambda^\varepsilon$, $\varepsilon>0$ introduced in Section \ref{main_sect}, and then  replacing the weight $\psi_s$ by its smooth approximations $\phi_{s,\varepsilon}$ bounded away from $0$ and so that $\phi_{s,\varepsilon}^{-1}$ is bounded. Now, however, if we define $\phi_{s,\varepsilon}$ by applying a standard (e.g.\,Friedrichs) 
mollifier to $\psi_s$, the task of evaluating $\phi_{s,\varepsilon} \Lambda^\varepsilon \phi_{s,\varepsilon}^{-1}f$ remains quite non-trivial.
We overcome this difficulty by considering a mollifier defined in terms of $\Lambda^\varepsilon$, see \eqref{reg_weight} below. 
This choice of the mollification is a key step in the proof.

In addition to the approximating operators $\Lambda_r^\varepsilon$, $\varepsilon>0$ in $L^r$, $r \in [1,\infty[$, we define in $C_u$
$$
\Lambda_{C_u}^\varepsilon:=-\varepsilon \Delta + (-\Delta)^{\frac{\alpha}{2}} - b_\varepsilon \cdot \nabla, \quad D(\Lambda_{C_u}^\varepsilon)=D((-\Delta)_{C_u}).
$$
Similarly to $\Lambda_r^\varepsilon$, for every $\varepsilon>0$ the operator $\Lambda_{C_u}^\varepsilon$ is
the generator of a positivity preserving contraction holomorphic semigroup (cf.\,\cite[Sect.\,8]{KS}).

We will also need
$$
(\Lambda^\varepsilon)_r^*:=-\varepsilon \Delta + (-\Delta)^{\frac{\alpha}{2}} + \nabla \cdot b_\varepsilon, \quad D(\Lambda_r^\varepsilon)=\mathcal W^{2,r}, \quad r \in [1,\infty[$$
$$
(\Lambda_\varepsilon)^*_{C_u}:=-\varepsilon \Delta + (-\Delta)^{\frac{\alpha}{2}} + \nabla \cdot b_\varepsilon, \quad D(\Lambda_{C_u}^\varepsilon)=D((-\Delta)_{C_u}).
$$
These are also generators of positivity preserving $L^\infty$ contraction quasi contraction holomorphic semigroups.
Moreover, there exists a constant $c_N$ independent of $\varepsilon$ such that, for  all $1 \leq r < q \leq \infty$,
\begin{equation}
\label{ultra}
\|e^{-t(\Lambda^\varepsilon)^*_r}\|_{r \rightarrow q} \leq c_N t^{-\frac{d}{\alpha}(\frac{1}{r}-\frac{1}{q})}, \quad t>0.
\end{equation}
Indeed, for $1 < r \leq q < \infty$ the ultracontractivity estimate follows from Claim \ref{claim0} by duality, and for all $1 \leq r \leq q \leq \infty$ upon taking limits $r \downarrow 1$, $q \uparrow \infty$.

In what follows, $s$ is fixed (since $\frac{s}{2} \leq t \leq s$, we have $s \leq 2$). We introduce the following two-parameter approximation of $\psi \equiv \psi_s$:
\begin{equation}
\label{reg_weight}
\phi_{n,\varepsilon}:=n^{-1} + e^{-\frac{(\Lambda^\varepsilon)^*}{n}}\psi \qquad (\varepsilon>0,\;n=1,2,\dots)
\end{equation}

\medskip

In $L^1$, define operators
\[
Q=\phi_{n,\varepsilon} \Lambda_1^\varepsilon \phi_{n,\varepsilon}^{-1}, \quad D(Q)=\phi_{n,\varepsilon} D(\Lambda^\varepsilon)
\]
and strongly continuous semigroups
$$e^{-tG}:=\phi_{n,\varepsilon} e^{-t\Lambda_1^\varepsilon}\phi_{n,\varepsilon}^{-1}.$$
Our goal is to show that $e^{-tG}$ satisfies
\begin{equation}
\label{accr3}
\|e^{-tG}f\|_1 \leq e^{(\hat{c}s^{-1}+\sigma_2+n^{-1})t}\|f\|_1, \quad t>0,
\end{equation}
so that we can pass to the limit (first in $\varepsilon$ and then in $n$) to establish \eqref{accr0}.
The difficulty is that a priori we have little information about $G$ to conclude \eqref{accr3}. On the other hand, we have detailed information about $Q$ and, moreover, intuitively $Q$ should coincide $G$. We prove this in Steps 1-3 below. 

\medskip

\textit{Step 1}.  Set
\begin{align*}
M:=&\,\phi_{n,\varepsilon}(1-\Delta)^{-1}[L^1 \cap C_{u}].
\end{align*}
This is a dense subspace of $L^1$ such that $$M\subset D(Q), \quad M\subset D(G)$$ and, furthermore, 
\begin{equation*}
Q\upharpoonright M\subset G.
\end{equation*}
(Indeed, for $f=\phi_{n,\varepsilon} u\in M$,
\[
Gf=s\mbox{-}L^1\mbox{-}\lim_{t\downarrow 0}t^{-1}(1-e^{-tG})f=\phi_{n,\varepsilon} s\mbox{-}L^1\mbox{-}\lim_{t\downarrow 0}t^{-1}(1-e^{-t\Lambda^\varepsilon})u=\phi_{n,\varepsilon} \Lambda^\varepsilon u=Qf\;).
\]
Thus $Q\upharpoonright M$ is closable and $$\tilde{Q}:=(Q\upharpoonright M)^{\rm clos}\subset G.$$
A standard argument shows that the range $\lambda_\varepsilon +\tilde{Q}$ is dense in $L^1$ (see \cite[Proof of Prop.\,1]{KS} for details).

\medskip

\textit{Step 2}. \textit{There are constants $\hat{c}>0$ and $\varepsilon_n>0$ such that, for every $n$ and all $0<\varepsilon 
\leq \varepsilon_n$,
  the operator $\lambda+\tilde{Q}$ is accretive whenever $\lambda\geq \hat{c} s^{-1}+ \sigma_2 + n^{-1}$,
i.e.
\begin{equation}
\label{accr}
\Real\langle (\lambda + \tilde{Q})f,\frac{f}{|f|}\rangle \geq 0 \quad \text{ for all } f \in D(\tilde{Q}),\end{equation}
where $s>0$ is from the definition of the weight $\phi_{n,\varepsilon}$. }

\medskip

Proof of \eqref{accr}.
We can represent $\psi \equiv \psi_s$ as $$\psi=\psi_{(1)} + \psi_{(u)}, \quad \text{where }0 \leq \psi_{(1)} \in D((-\Delta)_1), \qquad 0 \leq \psi_{(u)} \in D((-\Delta)_{C_{u}})$$
(e.g.\,$\psi_{(u)}:=1+\frac{\beta}{2}$ so $\psi_{(1)}$ has compact support and coincides with $s^{-\frac{\beta}{\alpha}}|x|^{\beta}$ around the origin). 
Therefore,
$$(\Lambda^\varepsilon)^*\psi=(\Lambda^\varepsilon)^*_{L^1}\psi_{(1)} + (\Lambda^\varepsilon)^*_{C_u}\psi_{(u)}$$ is well defined and belongs to $L^1 + C_u=\{w+v\mid w\in L^1, v\in C_u\}$.

By the construction of $\tilde{Q}$, it suffices to prove that
\begin{equation}
\label{accr2}
\Real\langle (\lambda + Q)f,\frac{f}{|f|}\rangle \geq 0 \quad \text{ for all } f \in M.
\end{equation} 
In what follows, we use the fact that both $e^{-t\Lambda^\varepsilon}$, $e^{-t(\Lambda^\varepsilon)^*}$ are holomorphic in $L^1$ and $C_u$. We have, for a $f=\phi_{n,\varepsilon} u$, $u \in (1-\Delta)^{-1}[L^1 \cap C_u]$,
$$
\langle Qf,\frac{f}{|f|}\rangle=\langle \phi_{n,\varepsilon} \Lambda^\varepsilon u,\frac{f}{|f|}\rangle=\lim_{t\downarrow 0}t^{-1}\langle\phi_{n,\varepsilon}(1-e^{-t \Lambda^\varepsilon})u,\frac{f}{|f|}\rangle,
$$
so
\begin{align*}
\Real\langle Qf,\frac{f}{|f|}\rangle & \geq\lim_{t\downarrow 0}t^{-1}\langle(1-e^{-t \Lambda^\varepsilon})|u|,\phi_{n,\varepsilon}\rangle \notag \\
&=\lim_{t\downarrow 0}t^{-1}\langle (1-e^{-t\Lambda^\varepsilon}) |u|,n^{-1}\rangle + \lim_{t\downarrow 0}t^{-1}\langle (1-e^{-t\Lambda^\varepsilon}) e^{-\frac{\Lambda^\varepsilon}{n}}|u|,\psi\rangle\\
&=\lim_{t\downarrow 0}t^{-1}\langle |u|,(1-e^{-t(\Lambda^\varepsilon)^*})n^{-1}\rangle+\lim_{t\downarrow 0}t^{-1}\langle e^{-\frac{\Lambda^\varepsilon}{n}}|u|,(1-e^{-t(\Lambda^\varepsilon)^*})\psi\rangle \\
&=\langle |u|, (\Lambda^\varepsilon)^*n^{-1}\rangle + \langle e^{-\frac{\Lambda^\varepsilon}{n}}|u|, (\Lambda^\varepsilon)^*\psi\rangle =: J_{1}+ J_{2},
\end{align*}
A simple calculation shows that ${\rm div\,}b_\varepsilon \geq -\sigma_2$ on $\mathbb R^d$ (cf.\,the proof of Claim \ref{claim1}) and so, since $\phi^{-1}_{n,\varepsilon} \leq n$, $$J_{1} \geq -\sigma_2 \|f\|_{1}.$$
We estimate $J_2$ using the next lemma. (It is in its proof that we use the fact that $|x|^\beta$ is a Lyapunov function of the formal operator $(-\Delta)^{\frac{\alpha}{2}} + \nabla \cdot \kappa|x|^{-\alpha}x$.)

\begin{lemma}
\label{aest_lem}
\begin{align*}
\Lambda^\varepsilon)^*\psi \geq -\hat{c}s^{-1}\psi - V_\varepsilon \quad \text{ on } \mathbb R^d,
\end{align*}
where $V_\varepsilon=
\varepsilon c_0\mathbf{1}_{|x| \leq 4^{1/\alpha}}|x|^{-2+\beta} + \mathbf{1}_{|x|<1}\kappa(d+\beta-\alpha)(|x|_\varepsilon^{-\alpha}-|x|^{-\alpha})|x|^{\beta} + c\mathbf{1}_{1\leq |x|\leq 2}|b_{\varepsilon}-b|$ for generic constants $\hat{c}$, $c_0$, $c$.
\end{lemma}

We will show below that the auxiliary potential $V_\varepsilon$ becomes negligible as $\varepsilon \downarrow 0$.

Lemma \ref{aest_lem} yields
$$
J_2 \geq - cs^{-1} \langle  e^{-\frac{\Lambda^{\varepsilon}}{n}}|u|, \psi \rangle- \langle  e^{-\frac{\Lambda^{\varepsilon}}{n}}|u|, V_{\varepsilon}\psi \rangle.
$$
Hence, taking into account the estimate on $J_1$,
\begin{align}
  \Real \langle Qf, \frac{f}{|f|} \rangle  &\geq  
  -\sigma_2 \|f\|_{1} - \hat{c}s^{-1} \langle |u|, e^{-\frac{(\Lambda^{\varepsilon})^{*}}{n}} \psi \rangle - \langle  e^{-\frac{(\Lambda^{\varepsilon})}{n}}|u|, V_{\varepsilon} \rangle \notag \\
& \text{(recall that $|u|=\phi_{n, \varepsilon}^{-1}|f|$  and $\phi_{n, \varepsilon} = n^{-1}+ e^{-\frac{(\Lambda^{\varepsilon})^{*}}{n}}\psi$ ) } \notag \\
& \geq -(cs^{-1}+\sigma_2 ) \|f\|_{1}-\langle |u|,e^{-\frac{(\Lambda^{\varepsilon})^{*}}{n}}( V_{\varepsilon}) \rangle. \tag{$\ast$}
\label{Qf}
\end{align}
By the ultracontractivity of $e^{-t(\Lambda^\varepsilon)^*}$, see \eqref{ultra}, and the fact that $\|V_\varepsilon\|_1 \downarrow 0$ as $\varepsilon \downarrow 0$, we have for every $n \geq 1$
\begin{align*}
\|e^{-\frac{(\Lambda^\varepsilon)^*}{n}}V_\varepsilon\|_\infty & \leq  c_N n^{\frac{d}{\alpha}}\|V_\varepsilon\|_1 \\
& (\text{we choose $\varepsilon_n>0$ such that for all $\varepsilon\leq \varepsilon_n$ $\|V_\varepsilon\|_1 \leq n^{-2}(c_N n^{\frac{d}{\alpha}})^{-1}$}) \\
& \leq  n^{-2}.
\end{align*}
Thus, since $\phi_{n,\varepsilon} \geq n^{-1}$, we have, for every $n=1,2,\dots$ and all $0<\varepsilon\leq \varepsilon_n$, $$\langle |u|,e^{-\frac{(\Lambda^{\varepsilon})^{*}}{n}}( V_{\varepsilon}\psi) \rangle \leq n^{-1}\|f\|_1.$$
Applying the latter in \eqref{Qf}, we obtain \eqref{accr2} $\Rightarrow$ \eqref{accr}.

\bigskip

\textit{Step 3}. Since $\tilde{Q}$ is closed and the range of $\lambda +\tilde{Q}$ is dense in $L^1$, the accretivitiy of $\lambda + \tilde{Q}$ in $L^1$ imply that the range of $\lambda_\varepsilon +\tilde{Q}$ is  in fact $L^1$ (see e.g.\,\cite[Appendix C]{KS}). Hence, by the Lumer-Phillips Theorem, $\lambda+\tilde{Q}$ is the generator of a contraction semigroup, and, since $\tilde{Q}\subset G$, we have
$$
\tilde{Q}=G.
$$ 

\smallskip

As a consequence of Steps 1-3, we obtain: for all $\varepsilon \leq \varepsilon_n$, $n=1,2,\dots$,
 \[
\label{star}
 \|e^{-tG}\|_{1\to 1}\equiv\|\phi_{n,\varepsilon} e^{-t \Lambda^\varepsilon}\phi_{n,\varepsilon}^{-1}\|_{1\to 1}\leq e^{(\hat{c} s^{-1} + \sigma_2 +n^{-1})t}. \tag{$\star$}
 \]
We pass to the limit in \eqref{star} in $\varepsilon \downarrow 0$ using Proposition \ref{prop1}, and then take $n \rightarrow \infty$. (See detailed argument in \cite{KS}.) This yields ($B_3$) and ends the proof of Theorem \ref{thm}. \hfill \qed

\bigskip

\section{Proof of Lemma \ref{aest_lem}}

Recall $\psi \equiv \psi_s$, $s \leq 2$.
We estimate the RHS of 
\begin{equation}
(\Lambda^{\varepsilon})^{*} \psi =  -\varepsilon \Delta \psi +(-\Delta)^{\frac{\alpha}{2}} \psi +{\rm div\,}( b_\varepsilon \psi)
\label{L_psi}
\end{equation}
in the next three claims. The first claim is straightforward:

\begin{claim}
\label{Cl1}
$-\varepsilon \Delta \psi \geq -P_\varepsilon,$
where $P_\varepsilon=
\varepsilon c_0\mathbf{1}_{|x| \leq 4^{1/\alpha}}|x|^{-2+\beta}$ for a generic constant $c_0$.
\end{claim}

To estimate the second term in \eqref{L_psi}, we introduce
$$
\tilde{\psi}(x):=s^{-\frac{\beta}{\alpha}}|x|^\beta.
$$
Clearly, $\psi$ and $\tilde{\psi}$ coincide in $B(0,s^{\frac{1}{\alpha}}),$
however, in contrast to $\psi$, the Lyapunov function $\tilde{\psi}$ grows at infinity.

\begin{claim}
\label{Cl2}
$(-\Delta)^{\frac{\alpha}{2}} \psi \geq - \beta(\beta-2+d)\frac{ \gamma(d+\beta-2)}{\gamma(d+\beta-\alpha)}|x|^{-\alpha}\tilde{\psi}$
\end{claim}

 \begin{proof}
We represent
  $
  (-\Delta)^{\frac{\alpha}{2}}h= -\Delta I_{2-\alpha} h=-I_{2-\alpha} \Delta h,
  $
where $I_{\nu}=(-\Delta)^{-\frac{\nu}{2}}$ is the Riesz potential. Then
 \begin{align*}
(-\Delta)^{\frac{\alpha}{2}}\psi = -I_{2-\alpha}\Delta \psi = -I_{2-\alpha}\Delta \tilde{\psi} -I_{2-\alpha}\Delta (\psi- \tilde{\psi})
\end{align*}
(all identities are in the sense of distributions).
We evaluate the first term in the RHS as
\begin{align*}
- I_{2-\alpha}\Delta \tilde{\psi} = - s^{-\frac{\beta}{\alpha}}\beta(d+\beta-2) I_{2-\alpha}|x|^{\beta-2}= -s^{-\frac{\beta}{\alpha}}\beta(d+\beta-2) \frac{\gamma(d+\beta-2)}{\gamma(d+\beta-\alpha)}|x|^{\beta-\alpha}
\end{align*}
and drop the second term since $-\Delta (\psi-\tilde{\psi}) \geq 0$ (see \cite[Remark 4]{KS} for the calculations).
\end{proof}

\begin{claim}
\label{Cl3}
$$   {\rm div\,}  ( b_{\varepsilon} \psi) \geq  {\rm div\,}( b\tilde{\psi}) -\hat{c}s^{-1} \psi   - U_{\varepsilon} - W_\varepsilon,$$ where $U_\varepsilon(x)=\mathbf{1}_{|x|<1}\kappa(d+\beta-\alpha)(|x|_\varepsilon^{-\alpha}-|x|^{-\alpha})|x|^\beta$ and $W_\varepsilon=c\mathbf{1}_{1\leq |x|\leq 2}|b_{\varepsilon}-b|$ for constants $\hat{c}$ and $c$.
\end{claim}

 \begin{proof}
 We represent
 $$
  {\rm div\,}  ( b_{\varepsilon} \psi) =  {\rm div\,} ( b\tilde{\psi})+ \big[  {\rm div\,} ( b_{\varepsilon} \psi)-  {\rm div\,} ( b\tilde{\psi}) \big].$$
	It is the difference $ {\rm div\,} ( b_{\varepsilon} \psi)-  {\rm div\,} ( b\tilde{\psi})$ that we need to estimate from below in terms of $U_{\varepsilon}\tilde{\psi}$ and $cs^{-1} \psi$.
We represent
 \begin{align*}
 [{\rm div\,} ( b_{\varepsilon} \psi)- {\rm div\,} ( b\tilde{\psi})]&=h_1+ {\rm div\,} \big[(b_{\varepsilon}-b)\psi \big], 
 \end{align*}
 where $h_1:={\rm div\,} \big [b(\psi- \tilde{\psi}) \big]$ is zero in $B(0, s^{\frac{1}{\alpha}})$, continuous and vanishes at infinity. (Indeed, on $\{|x| \geq 2\}$ 
$h_{1}=\kappa |x|^{-\alpha}x \nabla (\psi- \tilde{\psi}) + ({\rm div\,} b) (\psi- \tilde{\psi})$, where $|\nabla (\psi- \tilde{\psi})| \leq C_1|x|^{\beta-1}$, $\beta<\alpha$, while $|{\rm div\,} b| \leq C|x|^{-\alpha}$ by our assumption. Hence $h_1(x) \rightarrow 0$ as $x \rightarrow \infty$.) Moreover, a straightforward calculation shows that $$h_1 \geq -\hat{c}s^{-1}\psi.$$
In turn, we bound ${\rm div\,} \big[(b_{\varepsilon}-b)\psi \big]$ from below as follows:

1) On $\{|x| > 2\}$ we have $b_\varepsilon=b$, so ${\rm div\,} \big[(b_{\varepsilon}-b)\psi \big]=0$.

2) On $\{|x| < 1\}$,
\begin{align*}
{\rm div\,} \big[(b_{\varepsilon}-b)\psi \big] & = (b_{\varepsilon}-b)\cdot \nabla \psi  + ({\rm div\,} b_{\varepsilon} - {\rm div\,} b) \psi \\
& \geq \mathbf{1}_{|x|<1} \kappa( |x|_{\varepsilon}^{-\alpha}-|x|^{-\alpha})x \cdot \nabla  |x|^\beta + 
\mathbf{1}_{|x|<1}\kappa(d-\alpha)( |x|_{\varepsilon}^{-\alpha}-|x|^{-\alpha}) |x|^\beta \\
& = \mathbf{1}_{|x|<1}\kappa(d+\beta-\alpha)( |x|_{\varepsilon}^{-\alpha}-|x|^{-\alpha}) |x|^\beta.
\end{align*}

3) On $\{1 \leq |x|\leq 2\}$, $${\rm div\,} \big[(b_{\varepsilon}-b)\psi \big] \geq -c\mathbf{1}_{1 \leq |x| \leq 2}|b_{\varepsilon}-b|$$
for generic $c$.

Thus, everywhere on $\mathbb R^d$
$$
{\rm div\,} \big[(b_{\varepsilon}-b)\psi \big] \geq \mathbf{1}_{|x|<1}\kappa(d+\beta-\alpha)( |x|_{\varepsilon}^{-\alpha}-|x|^{-\alpha}) |x|^\beta -c\mathbf{1}_{1 \leq |x| \leq 2}|b_{\varepsilon}-b|,
$$
as needed.
 \end{proof}

Applying Claims \ref{Cl1}-\ref{Cl3} in \eqref{L_psi} and taking into account that, by our choice of $\beta$, $$- \beta(\beta-2+d)\frac{ \gamma(d+\beta-2)}{\gamma(d+\beta-\alpha)}|x|^{-\alpha}\tilde{\psi}+ {\rm div\,}( b\tilde{\psi})=0,$$ we obtain the assertion of the lemma with $V_\varepsilon:=P_\varepsilon+U_\varepsilon+W_\varepsilon$. \hfill \qed

\bigskip

\appendix

\section{Proof of Claim \ref{claim0}}

\label{claim0_proof}

The proof below follows closely e.g.\,\cite[Proof of Proposition 8]{KS} or \cite[Proof of Theorem 4.2]{KiS}.

Fix $\varepsilon>0$ and put $$u(t):=e^{-t\Lambda^\varepsilon}f, \quad f \in C_c^\infty,$$
where $\Lambda^\varepsilon=-\varepsilon\Delta + A - b \cdot \nabla$, $A:=(-\Delta)^{\frac{\alpha}{2}}$.
First, let $1<r<\infty$. Multiplying the equation $\partial_t u + \Lambda_r^\varepsilon u=0$ by $\bar{u}|u|^{r-2}$ and integrating in the spatial variables, we obtain
\begin{equation}
\label{pr}
\frac{1}{r}\partial_t\|u\|_r^r + \varepsilon\frac{4}{rr'}\|\nabla (u|u|^{\frac{r}{2}-1})\|_2^2  + \Real\langle Au,u|u|^{r-2}\rangle - \Real \langle b_\varepsilon \cdot \nabla u,u|u|^{r-2}\rangle=0.
\end{equation}
Since $-A$ is a Markov generator, we have using \cite[Theorem 2.1]{LS}
\begin{equation*}
\Real\langle Au,u|u|^{r-2}\rangle \geq \frac{4}{rr'}\|A^{\frac{1}{2}}u^{\frac{r}{2}}\|_2^2, \quad u^{\frac{r}{2}}:=u |u|^{\frac{r}{2}-1}.
\end{equation*}
Next, the integration by parts yields
\begin{align*}
 -\Real \langle b_\varepsilon \cdot \nabla u,u|u|^{r-2}\rangle  = \frac{1}{r}\langle {\rm div\,}b_\varepsilon, |u|^r\rangle,
\end{align*}
where on $\{|x|<1\}$ we have
\begin{align*}
{\rm div\,}b_\varepsilon = \kappa(d|x|_\varepsilon^{-\alpha}-\alpha|x|^{-\alpha-2}_\varepsilon|x|^2) \geq  \kappa(d-\alpha)|x|_\varepsilon^{-\alpha} >0,
\end{align*}
and on $\{|x| \geq 1\}$ $|{\rm div\,}b_\varepsilon| \leq \sigma_2$ by our assumption.
Therefore,
\begin{align*}
 -\Real \langle b_\varepsilon \cdot \nabla u,u|u|^{r-2}\rangle  
\geq - \frac{\sigma_2}{r}\langle |u|^r \rangle. 
\end{align*}
Thus, we obtain from \eqref{pr}
\begin{equation}
\label{i_5}
-\partial_t\|u\|_r^r\geq \frac{4}{r'}\|A^{\frac{1}{2}} u^{\frac{r}{2}}\|_2^2 - \sigma_2\|u\|_r^r.
\end{equation}
 From \eqref{i_5} we obtain $\|u(t)\|_r \leq e^{t\omega_r}\|f\|_r$ for appropriate $\omega_r>0$. Hence taking $r \downarrow 1$ and $r \uparrow \infty$, we obtain the first assertion of Claim \ref{claim0}, i.e.\,the quasi contractivitiy of $e^{-t\Lambda_r^\varepsilon}$ in $L^r$, $r \in [1,\infty[$ and its $L^\infty$ contractivity.

Let us prove the ultracontractivity of $e^{-t\Lambda_r^\varepsilon}$. By \eqref{i_5}, 
\[
-\partial_t\|u\|_{2r}^{2r} \geq \frac{4}{(2r)^\prime}\|A^{\frac{1}{2}} u^r\|_2^2 -  \sigma_2\|u\|_{2r}^{2r} , \quad 1\leq r<\infty. 
\]
Using the Nash inequality $\|A^\frac{1}{2} h \|_2^2 \geq C_N \|h\|_2^{2 + \frac{2\alpha}{d}} \|h\|_1^{-\frac{2\alpha}{d}}$ and $\|u(t)\|_r \leq e^{\omega_r t}\|f\|_r$,
integrating the previous inequality (see details e.g.\,in \cite[Proposition 8]{KS}, \cite[Theorem 4.2]{KiS}), we obtain
\begin{equation*}
\|e^{-t \Lambda^\varepsilon_r} \|_{r \rightarrow 2r} \leq c_3 e^{\omega_r t}  t^{-\frac{d}{\alpha} (\frac{1}{r} - \frac{1}{2r})}, \quad t > 0.
\end{equation*}
Now, using either the reproduction property or the Coulhon-Raynaud extrapolation (see e.g.\,\cite[Theorem F.1]{KiS}), we obtain the required ultracontractivity bound.

The previous argument yields: for $u\in D(\Lambda_2^\varepsilon)=\mathcal W^{2,2}$, $\Real\langle \Lambda^\varepsilon_2 u,u\rangle\geq \|A^\frac{1}{2}u\|_2^2$, 
so the fractional Sobolev Embedding Theorem now yields the required Sobolev embedding property. \hfill \qed

\end{document}